\documentclass{amsart}
\usepackage[margin=1in]{geometry}
\usepackage{amssymb, amsmath}
\usepackage{float}

 \usepackage[abs]{overpic}	
\usepackage{enumerate}

\usepackage{amsthm}

\usepackage{amssymb, amsmath}
\usepackage{float}

\usepackage[colorlinks = true,
            linkcolor = red,
            urlcolor  = red,
            citecolor = red,
            anchorcolor = red]{hyperref}
\usepackage{amsrefs}

\newtheorem{theorem}{Theorem}[section]
\newtheorem{proposition}[theorem]{Proposition}
\newtheorem{corollary}[theorem]{Corollary} 
\newtheorem{lemma}[theorem]{Lemma}

\newtheorem{definition}[theorem]{Definition}

\usepackage{algorithm2e,nicefrac}

\begin{document}

   \title{Vertex distortion detects the unknot}

   \author{Marion Campisi, Nicholas Cazet, David Crncevic, Tasha Fellman, Phillip Kessler, Nikolas Rieke, Vatsal Srivastava, and Luis Torres}

\maketitle

\begin{abstract}
The first two authors introduced vertex distortion of lattice knots and showed   that the vertex distortion of the unknot is 1. It was conjectured that the vertex distortion of a knot class is 1 if and only if it is trivial. We use Denne and Sullivan's lower bound on Gromov distortion to bound the vertex distortion of non-trivial lattice knots. This bounding allows us to conclude that a knot class has vertex distortion 1 if and only if it is trivial. We also show that vertex distortion does not have a universal upper bound and provide a vertex distortion calculator.

\end{abstract}


\section{Introduction }

The {\it Gromov distortion}  of a knot $K$  is denoted $\delta(K)$ and defined by \[ \delta(K):=\sup\limits_{a,b\in K} \frac{d_K(a,b)}{d(a,b)},\]  where $d_K(a,b)$ is the length of the shortest path from $a$ to $b$ in $K$ and $d$ is the Euclidean metric. Gromov defined distortion of rectifiable curves in \cite{gromov1983} where he showed that the Gromov distortion of any knot is at least $\pi/2$ and exactly $\pi/2$ if and only if the knot is a round circle. Denne and Sullivan later showed that the Gromov distortion of a non-trivial knot is non-trivially bounded from below.

  \begin{theorem}[Denne-Sullivan '04 \cite{Denne_2004}]
 \label{thm:den}
 For any non-trivial tame knot $K$, \[ \delta(K)\geq \frac{5\pi}{3}.\]
 \end{theorem}

Let $[K]$ be a knot class. The {\it Gromov distortion} of $[K]$ is denoted $\delta[K]$ and defined by \[ \delta[K]:= \inf\limits_{K\in[K]} \delta(K), \] where the infimum is taken over all conformation in the knot class. Gromov asked if there is a universal upper bound on this invariant. Pardon answered this negatively in  \cite{Pardon_2011}. He showed that the distortion of the $(p,q)$-torus knot class tends to infinity as $p$ and $q$ tend to infinity.  In \cite{Blair_2020},  Blair, Campisi, Taylor, and Tomova showed that Gromov distortion is bounded below by bridge number and bridge distance, and exhibited an infinite family of knots for which their bound is arbitrarily stronger than Pardon's. No knot class has a calculated Gromov distortion.

 The first two authors introduced vertex distortion of lattice knots as a discrete reformulation of Gromov distortion \cite{Campisi_2021}. A {\it polygonal knot} is a knot that consists of finitely many line segments called {\it sticks}.  A {\it lattice knot} is a polygonal knot in the cubic lattice $\mathbb{L}^3=(\mathbb{R}\times\mathbb{Z}\times\mathbb{Z})\cup(\mathbb{Z}\times\mathbb{R}\times\mathbb{Z})\cup (\mathbb{Z}\times\mathbb{Z}\times\mathbb{R}).$  For a lattice knot $K$, the points ${\bf a}\in V(K):=K\cap\mathbb{Z}^3$ are called {\it vertices}; bold variables will represent vertices. The {\it vertex distortion} of a lattice knot $K$ is defined as

\[
\delta_V(K):=\max\limits_{{\bf a,b}\in V(K)} \frac{d_{K}({\bf a,b})}{d_1({\bf a,b})},
\]
 where $d_{K}({\bf a,b})$ is the length of the shortest path from ${\bf a}$ to ${\bf b}$ in $K$ and $d_1({\bf a,b})=\sum_{i=1}^3 |a_i-b_i|$.
The {\it vertex distortion} of a knot class $[K]$ is given by
\[
\delta_V[K]:=\inf\limits_{K \in [K]} \delta_V(K),
\]
where the infimum is taken over all lattice knot conformations representing the knot class $[K]$.

In \cite{Campisi_2021}, vertex distortion was shown  to satisfy many of the same properties as Gromov distortion: 
{\it 
\begin{itemize}
\item[(i)] If $\delta_V(K)=1$, then $K$ is unknotted.

\item[(ii)] Up to isometry, the only two conformation with vertex distortion 1 are shown in Figure \ref{fig:triv}.
\item[(iii)] There exists a sequence of  minimal stick $(n,n+1)$- torus knots $T_n$ such that $\delta_V(T_n)\to\infty$ as $n\to\infty$.
\item[(iv)] If $U$ is the unknot, then $\delta_V[U]=1.$ 

\end{itemize}
 }

 \begin{figure}
 \centering
  \includegraphics[scale=.6]{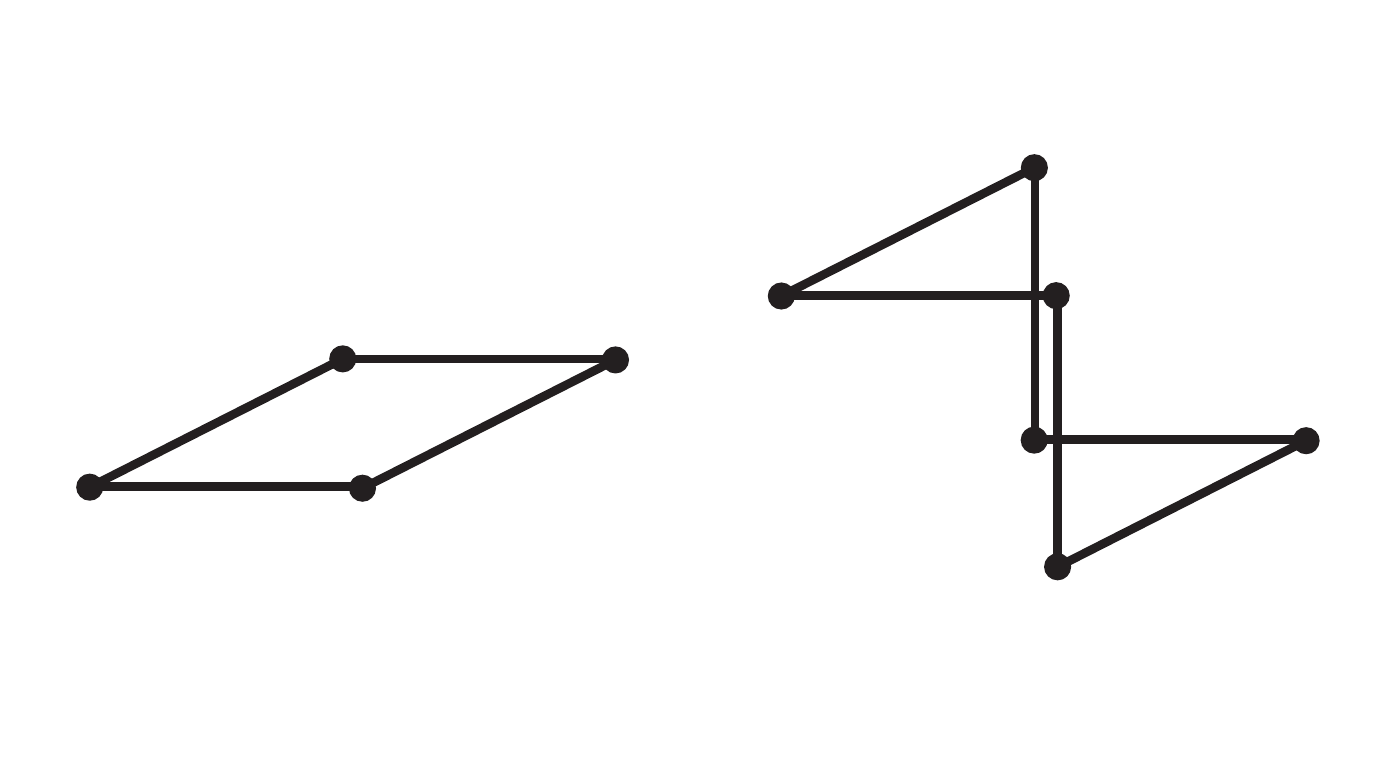}
 \label{fig:triv}
\caption{Conformations with vertex distortion 1.}
 \end{figure}
 
 \noindent It was conjectured that $\delta_V[K]=1$ if and only if $[K]$ is trivial and \[ \delta_V[T_{p,q}]\to\infty \ \text{as} \ p,q\to\infty.\]

The goal of this paper is to affirm these two conjectures. These conjectures are corollaries to the two theorems of this paper:

   \begin{theorem}
   
   If $K$ is a non-trivial lattice knot, then \[ \delta_V(K)\geq\frac{5\pi}{3\sqrt3}-1\approx 2.02.\]
   \label{fig:bound}
   \end{theorem}

\begin{theorem}
\label{prop:bounding1}
For any lattice knot $K$, 
\[ \delta(K)\geq \delta_V(K)\geq \delta(K)/\sqrt3-1.\]

\end{theorem}

 Section \ref{section:defns} gives definitions and technical lemmas on vertex distortion needed to prove Theorem \ref{fig:bound} and Theorem \ref{prop:bounding1}.  This section compares the vertex distortion of $K$ with the distortion ratios between midpoints of $K$'s edges. A notion of generality is defined for pairs of midpoints. Section \ref{section:defns} culminates with Proposition \ref{prop:bigratio} that shows if two midpoints have a higher distortion ratio than the vertex distortion of the lattice knot that they belong to, then the midpoints are antipodal and non-generic. 
 
 Section \ref{double} defines the $2^n$-scaling of a lattice knot. The doubling of a lattice knot is done by multiplying the coordinates of each point by 2. This section also defines Gromov 1-distortion by replacing the Euclidean metric $d$ with the $d_1$ metric in the definition of Gromov distortion. Section \ref{double} culminates with Proposition \ref{lem:3} showing that the Gromov 1-distortion of a lattice knot is to equal the vertex distortion of its double.

  Section \ref{section:proof} proves  Theorem \ref{fig:bound}. The proof translates Denne and Sullivan's  bound on Gromov distortion to vertex distortion using Gromov 1-distortion and the results of Sections \ref{section:defns} and \ref{double}. A corollary is included that affirms $\delta_V[K]=1$ if and only if $[K]$ is trivial.
  
 Section \ref{section:misc} proves Theorem \ref{prop:bounding1}, affirms \[ \delta_V[T_{p,q}]\to\infty \ \text{as} \ p,q\to\infty,\]    and gives a link to our vertex distortion calculator.

\section{Vertex Distortion and Midpoints}
\label{section:defns}

This section includes midpoints of edges in its study of vertex distortion. A notion of generality is defined for pairs of midpoints such that generic midpoints have a distortion ratio no larger than the vertex distortion of their knot. The lemmas of this section are used to prove Proposition \ref{prop:bigratio} that says if a pair of midpoints have a higher distortion ratio than the vertex distortion of their knot, then they are antipodal and non-generic. This classification is used in the proof of Theorem \ref{fig:bound}.

\begin{definition} A {\it stick} of a lattice knot $K$ is a maximal line segment of $K$ in $\mathbb{L}^3$, i.e. a line segment contained in $K$ that is not contained in any longer line segment of $K$. \end{definition}

Each stick of length $n$ can be decomposed into $n$ unit length line segments.

\begin{definition} An {\it edge} of a lattice knot is a unit length line segment connecting two vertices. \end{definition}

A stick  parallel to the $x$-, $y$-, or $z$-axis is called an $x$-, $y$-, or $z$-stick, respectively. We define an $x$-, $y$-, and $z$-edge analogously.

\begin{lemma}

A lattice knot has an even number of edges. 
\label{lem:even}
\end{lemma}

\begin{proof}
Orient the lattice knot. Consider an $x$-edge to be positive if the $x$-coordinate increases following the orientation and negative if the $x$-coordinate decreases. There is an equal number of positive $x$-edges as negative $x$-edges since the lattice knot is a closed loop. This is true for $y$- and $z$-edges as well. Therefore, there is an equal number of negative edges as positive edges in the lattice knot. Since each edge is either positive or negative, the total number of edges is divisible by 2. 
\end{proof}

 \begin{definition}
A point $a^*$ on a knot $K$ is {\it antipodal} to  $a\in K$  if $d_K(a,a^*)$ is half the total length of $K$.
 
 \end{definition}   
 
 \begin{definition}
A point on a lattice knot $K$ is a {\it midpoint} if it is the midpoint of some edge; the vertices of this edge are called the {\it neighbor vertices} or {\it neighbors} of the midpoint.

\end{definition}

\begin{definition}

 Let $V_m(K)$ be the set of vertices and midpoints of a lattice knot $K$.
 
 \end{definition}

We will denote the neighbors of a midpoint $p$ by ${\bf n_p^-}$ and ${\bf n_p^+}$.

    \begin{definition}
    Let $p$ and $q$ be midpoints of a lattice knot with neighbors ${\bf n_p^-},{\bf n_p^+}$ and ${\bf n_q^-},{\bf n_q^+}$. The pair $p$ and $q$ are  {\it generic} if $d_1(p,{\bf n_q^-})\neq d_1(p,{\bf n_q^+})$ or $d_1({\bf n_p^-},q)\neq d_1({\bf n_p^+},q)$.
    
    \end{definition}

    \begin{figure}[ht]
\centering

\includegraphics[scale=.8]{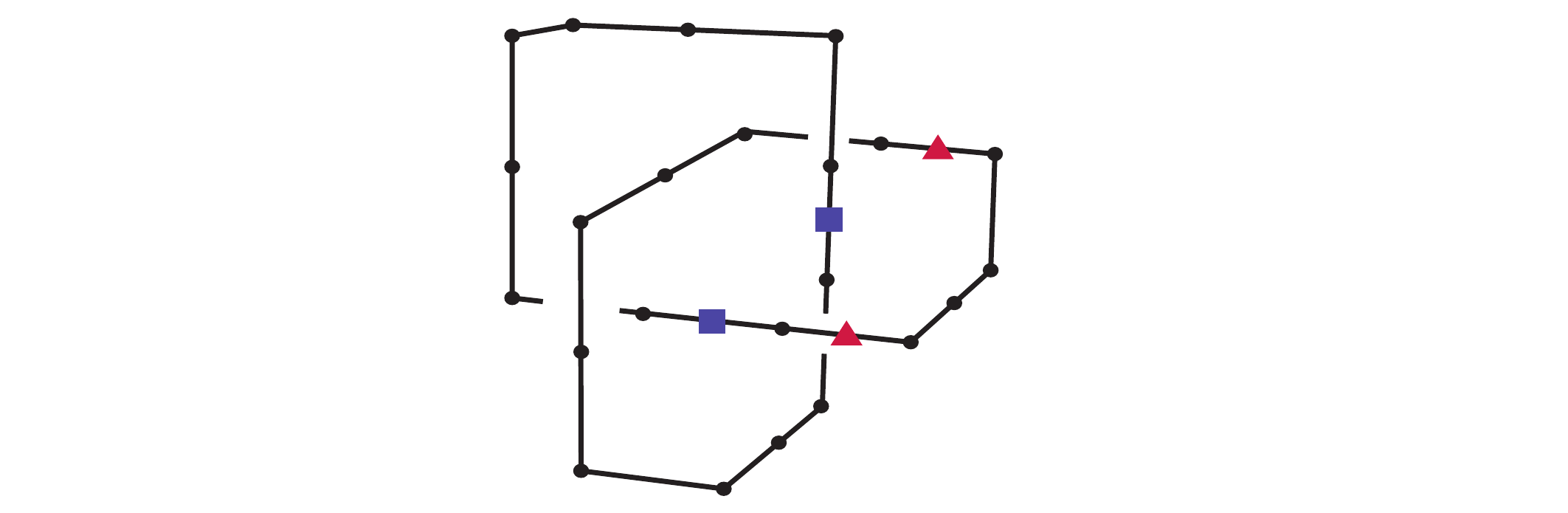}

\caption{A pair of generic midpoints, squares, and non-generic midpoints, triangles.}
\label{fig:generic}
\end{figure}

  \begin{lemma} 

If two midpoints $p$ and $q$ of a lattice knot $K$ are non-generic, then they share a non-integer coordinate, the edges containing $p$ and $q$ are parallel, and   \[d_1(p,{\bf n_q^-})=d_1(p,{\bf n_q^+})=d_1(p,q)+0.5=d_1({\bf n_p^-},q)=d_1({\bf n_p^+},q). \] 
\label{lem:nongeneric}

\end{lemma}

\begin{proof}

  For any pair of midpoints $p$ and $q$ and neighbors ${\bf n_p}$ and ${\bf n_q}$, we have that  \[ d_1(p,{\bf n_q})=d_1(p,q)+0.5 \quad \text{or} \quad d_1(p,{\bf n_q})=d_1(p,q)-0.5,\]    
    and 
    \[ d_1({\bf n_p},q)=d_1(p,q)+0.5 \quad \text{or} \quad d_1({\bf n_p},q)=d_1(p,q)-0.5.\]   
     
    Therefore, a pair of midpoints $p$ and $q$ are non-generic if and only if     
  \[d_1(p,{\bf n_q^-})=d_1(p,{\bf n_q^+})=d_1(p,q)+0.5 \] and   \[d_1({\bf n_p^-},q)=d_1({\bf n_p^+},q)=d_1(p,q)+0.5,\] since it is not possible for $d_1(p,{\bf n_q^-})=d_1(p,{\bf n_q^+})=d_1(p,q)-0.5=d_1({\bf n_p^-},q)=d_1({\bf n_p^+},q).$
  
   Suppose that $p=(p_x,p_y,p_z)$ and $q=(q_x,q_y,q_z)$. The neighbor vertices of $q$ are $(q_x\pm0.5,q_y,q_z), (q_x,q_y\pm0.5,q_z),$ or $(q_x,q_y,q_z\pm0.5)$. In the first case,  $d_1(p,{\bf n_q^-})= d_1(p,{\bf n_q^+})$ if and only if $p_x=q_x$.  With neighbor vertices $(q_x\pm0.5,q_y,q_z)$ and $p_x=q_x$, we have that $d_1(p,q)=|p_y-q_y|+|p_z-q_z|$ and $d_1(p,{\bf n_q^-})=d_1(p,{\bf n_q^+})=0.5+|p_y-q_y|+|p_z-q_z|=d_1(p,q)+0.5$. In the case that the neighbor vertices of $q$ are  $(q_x,q_y\pm0.5,q_z)$ or $(q_x,q_y,q_z\pm0.5)$, an analogous argument shows that $d_1(p,{\bf n_q^-})= d_1(p,{\bf n_q^+})$ only if  $d_1(p,{\bf n_q^-})=d_1(p,{\bf n_q^+})=d_1(p,q)+0.5$. Likewise, $d_1(q,{\bf n_p^-})= d_1(q,{\bf n_p^+})$ only if  $d_1(q,{\bf n_p^-})=d_1(q,{\bf n_p^+})=d_1(p,q)+0.5$.

The non-integer coordinate of a midpoint determines whether it belongs to an $x$-, $y$-, or $z$-edge. Since non-generic midpoints share a non-integer coordinate, the edges of non-generic midpoints are parallel. 

\end{proof}

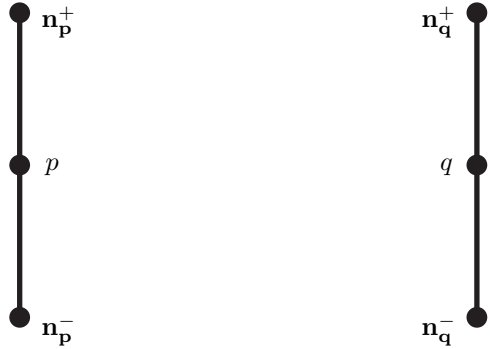
\begin{figure}[ht]
\centering

\begin{overpic}[unit=.5mm,scale=1]{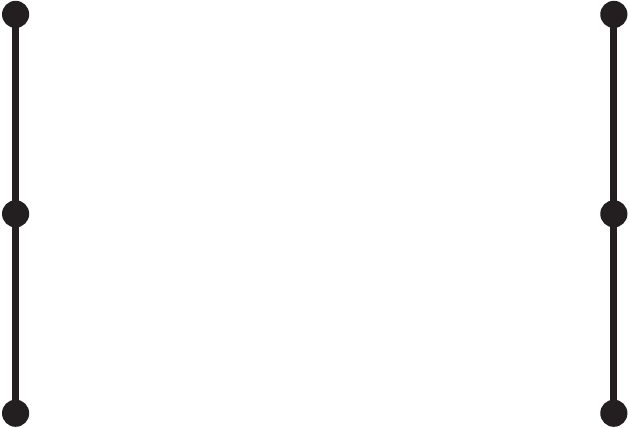} \put(10,42){$p$} \put(9,80){${\bf n_p^+}$}  \put(9,-2){${\bf n_p^-}$} \put(115,42){$q$}  \put(110,80){${\bf n_q^+}$}  \put(110,-2){${\bf n_q^-}$} 
\end{overpic}
\vspace{1cm}
\caption{Projection of non-generic midpoints.}
\label{fig:nongeneric}
\end{figure}

We use distortion functions  to translate the bound of Theorem \ref{thm:den}.

\begin{definition}
    For a knot $K$,  define $\rho_1, \rho_2: K \times K \to \mathbb{R}$ by 
    \[
        \rho_1(x, y) = \frac{d_K(x, y)}{d_1(x, y)} 
    \]

    \[
        \rho_2(x, y) = \frac{d_K(x, y)}{d(x, y)} 
    \]
   for $x\neq y$ and $\rho_1(x,x)=\rho_2(x,x)=1$. \end{definition}

\begin{lemma}
    \label{lem:nearvertex}
   For any midpoint $m=(m_1,m_2,m_3)$ and vertex ${\bf v}=(v_1,v_2,v_3)$ of a lattice knot, there exists a neighbor ${\bf n_m}$ of $m$ such that
    \[
        \rho_1({\bf n_m}, {\bf v}) \geq \rho_1(m, {\bf v}).
    \]
\end{lemma}

\begin{proof}
    Let ${\bf n_m^-}$ and ${\bf n_m^+}$ be the neighbors of $m$. Without loss of generality, assume that $m$ is a midpoint of an $x$-edge and 
    \[
        {\bf n_m^-} = (m_1-0.5, m_2, m_3) \quad\quad\quad  {\bf n_m^+} = (m_1+0.5, m_2, m_3).
    \]
     Since $m$ is the midpoint of an $x$-edge,  we have $v_1 \neq m_1$. Figure \ref{fig:lemma} gives intuition for the following cases.
     
     If $v_1 > m_1$, then
    \begin{equation}
        \label{b1>a1}
        \begin{split}
            d_1({\bf n_m^-, v}) = d_1(m, {\bf v}) + 0.5
        \end{split}
        \quad\quad\quad
        \begin{split}
            d_1({\bf n_m^+, v}) = d_1(m, {\bf v}) - 0.5.
        \end{split}
    \end{equation}
    
  If  $v_1 < m_1$, then
    \begin{equation}
        \label{b1<a1}
        \begin{split}
            d_1({\bf n_m^-, v}) = d_1(m, {\bf v}) - 0.5
        \end{split}
        \quad\quad\quad
        \begin{split}
            d_1({\bf n_m^+, v}) = d_1(m, {\bf v}) + 0.5.
        \end{split}
    \end{equation}

\begin{figure}

\centering

\begin{overpic}[unit=.5mm,scale=.85]{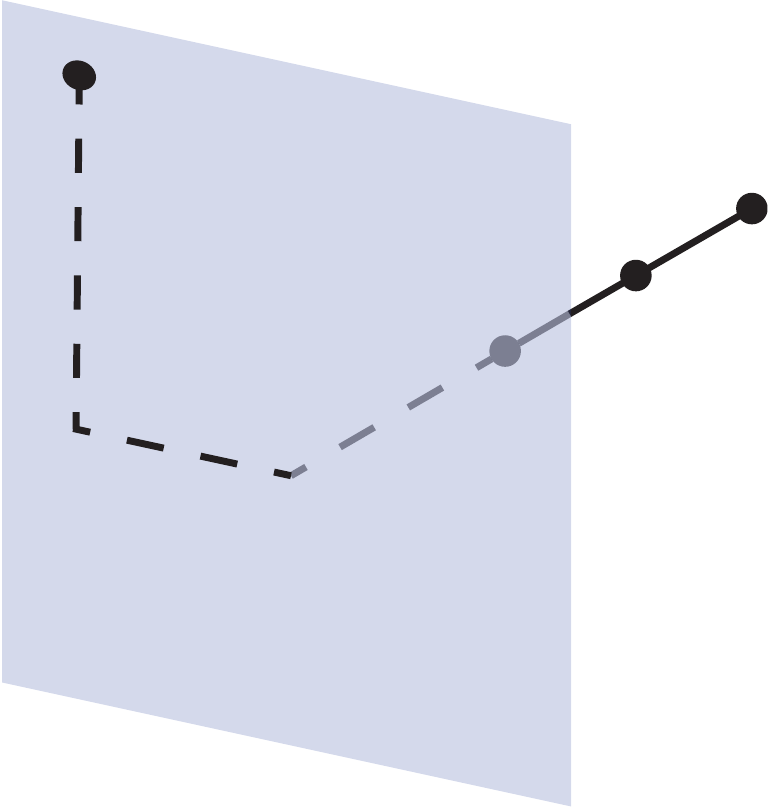} \put(18,122){$\bf v$} \put(65,20){$x=v_1$} \put(110,83){$m$}\put(130,95){$\bf n_m^-$}\put(88,71){$\bf n_m^+$}

\end{overpic}
\caption{The case $v_1>m_1$ seen in  (\ref{b1>a1}).}
\label{fig:lemma}
\end{figure}

    Lemma \ref{lem:even} tells us that the total length of a lattice knot is even. Therefore, the $d_K$-distance between antipodal points on a lattice knot is an integer, and the antipode of a vertex is a vertex. For the vertex ${\bf v}$, consider its antipodal vertex ${\bf v}^*\neq m$. The antipode ${\bf v}^*$ lies on one of the two injective paths ${\bf v} \to {\bf n_m^-} \to m$ or ${\bf v} \to {\bf n_m^+} \to m$ along $K$.

  Case 1:  ${\bf v}^*$ lies on the path ${\bf v} \to {\bf n_m^-} \to m$. Then,
        \[
            \begin{split}
                d_K({\bf n_m^-}, {\bf v}) = d_K(m, {\bf v}) + 0.5
            \end{split}
            \quad\quad\quad
            \begin{split}
                d_K({\bf n_m^+}, {\bf v}) = d_K(m, {\bf v}) - 0.5.
            \end{split}
        \]

      If $v_1 > m_1$, then by  (\ref{b1>a1})
        \[
            \rho_1({\bf n_m^+,v}) = \frac{d_K(m,{\bf v})-0.5}{d_1(m,{\bf v})-0.5} \geq \frac{d_K(m,{\bf v})}{d_1(m,{\bf v})}=\rho_1(m,{\bf v}),
        \] since  $d_K(m,{\bf v})\geq d_1(m,{\bf v}).$
        
            If $v_1 < m_1$, then by (\ref{b1<a1})
        \[
            \rho_1({\bf n_m^-,v}) = \frac{d_K(m,{\bf v})+0.5}{d_1(m,{\bf v})-0.5} \geq \rho_1(m,{\bf v}).
        \]
Case 2: ${\bf v}^*$ lies on the path ${\bf v} \to {\bf n_m^+} \to m$. Then,
        \[
            \begin{split}
                d_K({\bf n_m^-, v}) = d_K(m, {\bf v}) - 0.5
            \end{split}
            \quad\quad\quad
            \begin{split}
                d_K({\bf n_m^+, v}) = d_K(m, {\bf v}) + 0.5.
            \end{split}
        \]
        
             If $v_1 > m_1$, then by (\ref{b1>a1})
        \[
            \rho_1({\bf n_m^+,v}) = \frac{d_K(m,{\bf v})+0.5}{d_1(m,{\bf v})-0.5} \geq \rho_1(m,{\bf v}).
        \]
        
      If $v_1 < m_1$, then by  (\ref{b1<a1})
        \[
            \rho_1({\bf n_m^-,v}) = \frac{d_K(m,{\bf v})-0.5}{d_1(m,{\bf v})-0.5} \geq \rho_1(m,{\bf v}),
        \] since $d_K(m,{\bf v})\geq d_1(m,{\bf v}).$

\end{proof}

    \begin{lemma}
    \label{lem:generic}
    If two midpoints $p$ and $q$ of a lattice knot are generic, then there exist neighbor vertices ${\bf n_p}$ of $p$ and  ${\bf n_q}$ of $q$ such that
    \[
        \rho_1({\bf n_p, n_q}) \geq \rho_1(p, q).
    \]
\end{lemma}

\begin{proof}

Assume that $d_1(p,{\bf n_q^+})\neq d_1(p,{\bf n_q^-})$; an analogous argument can used if $d_1({\bf n_p^+},q)\neq d_1({\bf n_p^-},q)$. Label the neighbors ${\bf n_q^\pm}$ of $q$ such that $d_1(p, {\bf n_q^-} )=d_1(p,q)-0.5$ and $d_1(p,{\bf n_q^+})=d_1(p,q)+0.5$. 
 
Either $d_K(p,{\bf n_q^-})=d_K(p,q)-0.5$ or $d_K(p,{\bf n_q^-})=d_K(p,q)+0.5$. Since \[ \frac{d_K(p,q)+0.5}{d_1(p,q)-0.5}> \frac{d_K(p,q)-0.5}{d_1(p,q)-0.5}\geq \rho_1(p,q),\] we have that
    \[
        \rho_1(p,{\bf n_q^-}) \geq \rho_1(p,q).
    \]
    
Then, Lemma \ref{lem:nearvertex}  implies that there exists  ${\bf n_p} \in \{ {\bf n_p^+,n_p^- } \} $ such that
 \[
        \rho_1({\bf n_p},{\bf n_q^-}) \geq \rho_1(p,q).
    \]

\end{proof}

\begin{lemma}

If two midpoints $p$ and $q$ of an oriented lattice knot are non-generic and non-antipodal, then there exist neighbor vertices ${\bf n_p}$ of $p$ and ${\bf n_q}$ of $q$ such that \[ \rho_1({\bf n_p, n_q})\geq \rho_1(p,q).\]
\label{lem:notgeneric}
\end{lemma}

\begin{proof}

Let ${\bf n_p^\pm}$ and ${\bf n_q^\pm}$  be the neighbors of $p$ and $q$. The lattice knot's orientation  induces one of four orderings on these neighbors: \begin{equation} {\bf n_p^-}\to {\bf n_p^+} \to {\bf n_q^+}\to {\bf n_q^-}, \label{3}\end{equation} 
\begin{equation}{\bf n_p^-}\to {\bf n_p^+} \to {\bf n_q^-}\to {\bf n_q^+} ,\label{4}\end{equation}
 \begin{equation*}  {\bf n_p^-}\leftarrow {\bf n_p^+} \leftarrow {\bf n_q^+}\leftarrow{\bf n_q^-},\end{equation*}
  \begin{equation*}  {\bf n_p^-}\leftarrow {\bf n_p^+} \leftarrow {\bf n_q^-}\leftarrow{\bf n_q^+}.\end{equation*}
 Since $\rho_1$ is unaffected by orientation reversal, we need only consider the cases (\ref{3}) and (\ref{4}).

With $p$ and $q$ non-antipodal, there are two subcases to consider: the antipode of $p$ lies on the lattice knot path $p\to {\bf n_p^+}\to q$ or $p\to {\bf n_p^-}\to q$. 

Case 1: The antipode of $p$ lies on $p\to {\bf n_p^+}\to q$ . If (\ref{3}) holds, then \[ \rho_1({\bf n_p^+, n_q^+})> \rho_1(p,q).\] 
If (\ref{4}) holds, then \[\rho_1({\bf n_p^+,n_q^+})=\rho_1(p,q). \]

Case 2: The antipode of $p$ lies on $p\to {\bf n_p^-}\to q$. If (\ref{3}) holds, then \[ \rho_1({\bf n_p^-, n_q^-})>\rho_1(p,q).\] If (\ref{4}) holds, then \[\rho_1({\bf n_p^-,n_q^-})=\rho_1(p,q). \] 

\end{proof}

\begin{proposition}

    If $p$ and $q$ are midpoints of a lattice knot $K$ such that
    \[\rho_1(p, q) > \delta_V(K), \]
    then $p$ and $q$ are antipodal and non-generic. Furthermore, the neighbors of $p$ and $q$ satisfy the ordering of (\ref{3}) for some orientation of $K$.
        \label{prop:bigratio}
\end{proposition}

\begin{figure}[ht]
\centering
\vspace{5mm}

\begin{overpic}[unit=.5mm,scale=1]{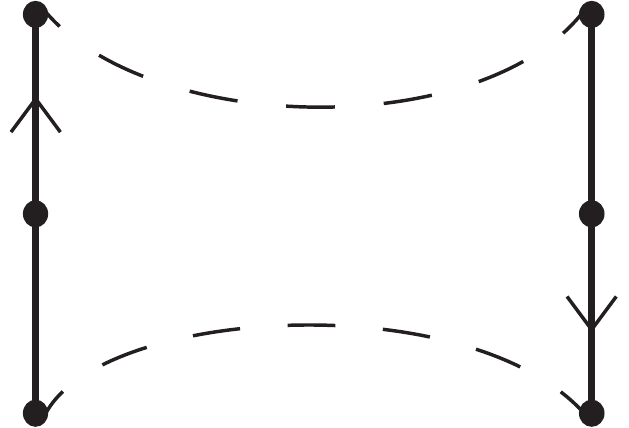} \put(-2,42){$p$} \put(-5,90){${\bf n_p^+}$}  \put(-5,-6){${\bf n_p^-}$} \put(126,42){$q$}  \put(123,90){${\bf n_q^+}$}  \put(123,-6){${\bf n_q^-}$} \put(62,70){$\alpha$}\put(62,11){$\beta$}
\end{overpic}
\vspace{1cm}
\caption{Neighbor ordering induced by (\ref{3}).}
\label{fig:nongeneric2}
\end{figure}
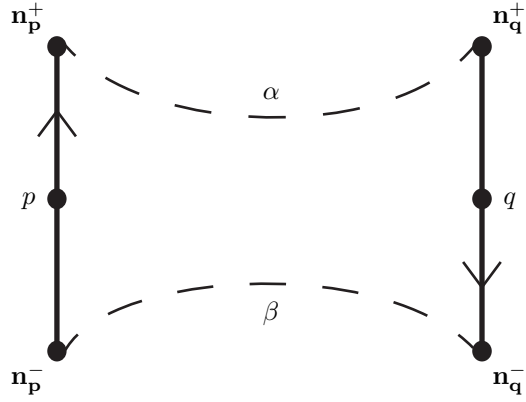

\begin{proof}
If $p$ and $q$ are generic, then Lemma \ref{lem:generic} implies $\delta_V(K)\geq \rho_1(p,q)$. If $p$ and $q$ non-antipodal (and non-generic), then Lemma \ref{lem:notgeneric} gives $\delta_V(K)\geq \rho_1(p,q)$. Therefore, $p$ and $q$ are antipodal and non-generic.

Recall from the proof of Lemma \ref{lem:notgeneric} that the neighbors of $p$ and $q$ satisfy the ordering of (\ref{3}) or (\ref{4}). If (\ref{4}) is satisfied, then there exist neighbors ${\bf n_p}$ and ${\bf n_q}$ such that  $\rho_1({\bf n_p, n_q})=\rho_1(p, q).$ Thus, non-generic, antipodal $p$ and $q$ satisfy $\rho_1(p,q)>\delta_V(K)$ only if (\ref{3}) or its reverse is satisfied. 

\end{proof}


\section{$2^n$-Scaled Lattice Knots and Gromov 1-distortion}
\label{double}

This section introduces $2^n$-scaling of lattice knots. A lattice knot is doubled by multiply the coordinates of each point by 2. We prove in this section that the vertex distortion of a doubled lattice knot equals the vertex distortion of its progressive doublings. Gromov 1-distortion is defined by replacing the Euclidean metric $d$ with the $d_1$ metric. This section shows that the Gromov 1-distortion of a lattice knot equals the vertex distortion of its double. This result is used to prove Theorem \ref{fig:bound} in the next section.

\begin{definition}

The {\it $2^n$-scaling} of a lattice knot $K$ is \[ 2^n K=\{ (2^nx,2^ny,2^nz) \ | \ (x,y,z)\in K\},\]
where $n\in\mathbb{N}$.
\end{definition}

 For $r\in\mathbb{R}^3$, we have that $r\in \mathbb{L}^3$ if and only if at least two coordinates of $r$ are integers. Therefore, the integer scaling of $\mathbb{L}^3$ is contained in $\mathbb{L}^3$. This implies that a $2^n$-scaled lattice knot is a lattice knot.

\begin{lemma}
\label{lem:bigger}
For a lattice knot $K$, if $\delta_V(2K)>\delta_V(K)$ then there exists antipodal, non-generic midpoints $p$ and $q$ of $K$ such that \[\rho_1(p,q)=\delta_V(2K).\] Furthermore, the neighbor vertices of $p$ and $q$ satisfy the ordering of (\ref{3}) illustrated in Figure \ref{fig:nongeneric2}.
\end{lemma}

\begin{proof}
Since $\rho_1(2x,2y)=\rho_1(x,y)$, we have that $\delta_V(2K)$ is equal to \[ \max_{p,q\in V_m(K)} \rho_1(p,q).\] The inequality $\delta_V(2K)> \delta_V(K)$ is equivalent to \begin{equation} \max_{p,q\in V_m(K)} \rho_1(p,q) > \max_{{\bf a},{\bf b}\in  V(K)} \rho_1({\bf a},{\bf b}). 
\label{inequal}\end{equation}

The inequality (\ref{inequal}) holds if and only if there is a pair $p',q'\in V_m(K)$ such that $\rho_1(p',q')=\delta_V(2K)>\delta_V(K)$. At least one of $p'$ and $q'$ must be a midpoint of $K$ for this to be true. Lemma \ref{lem:nearvertex} tells us that if only one of $p'$ and $q'$ is a midpoint, then $\rho_1(p',q')\leq\delta_V(K).$ Therefore, both $p'$ and $q'$ must be midpoints for the inequality (\ref{inequal}) to hold. With $p'$ and $q'$ both midpoints of $K$ such that $\rho_1(p',q')>\delta_V(K)$, Proposition \ref{prop:bigratio} tells us that $p'$ and $q'$ are antipodal and non-generic. The proposition further implies that the neighbor vertices of $p'$ and $q'$ satisfy the ordering of (\ref{3}). This is shown in Figure \ref{fig:nongeneric2}.

\end{proof}

\begin{lemma}
\label{lem:lim2}
For any lattice knot $K$ and positive integer $n$,  \[\delta_V(2^{n+1}K)=\delta_V(2^nK).\] 
\end{lemma}

\begin{proof}
We will show that \[\delta_V(4K)=\delta_V(2K).\] The lemma follows from a recursion of this fact. 

Assume that $\delta_V(4K)> \delta_V(2K).$  Lemma \ref{lem:bigger} implies that there exist antipodal, non-generic midpoints $p$ and $q$ of $2K$ such that $\rho_1(p,q)=\delta_V(4K)>\delta_V(2K)$. The neighbor vertices of $p$ and $q$ satisfy the ordering of (\ref{3}) illustrated in Figure \ref{fig:nongeneric2}.

Let $\alpha$ and $\beta$ be the embedded paths in $K$ connecting ${\bf n_p^+}$ to ${\bf n_q^+}$ and ${\bf n_p^-}$ to ${\bf n_q^-}$. These are abstractly drawn in Figure \ref{fig:nongeneric2}. The paths $\alpha$ and $\beta$ have equal length $\ell\in\mathbb{N}$ since $p$ and $q$ are antipodal. The total length of the knot is $2\ell+2$. Lemma \ref{lem:even} tells us that lattice knots have even length implying that four divides the arc length of $2K$. Therefore, $\ell$ must be odd.

Connect the endpoints of $\alpha$ with a path $L$ in $\mathbb{L}^3$ with arc length equal to the $d_1$-distance between the endpoints. Let $|L|$ denote this arc length. The closed loop $L\cup \alpha$ has even length $|L|+\ell$ if and only if $|L|$ is odd. Therefore, $|L|$ must be odd.
  
  If   ${\bf n_p^+}$ and ${\bf n_q^+}$ have even coordinates, then $|L|$ is even, forming a contradiction. Following the orientation of the knot, every other vertex of $2K$ has even coordinates since it is the double a vertex of $K$. Therefore, if neither ${\bf n_p^+}$ nor ${\bf n_q^+}$ has all even coordinates, then ${\bf n_p^-}$ and ${\bf n_q^-}$ have even coordinates. With a minimal length path in lattice chosen to connect $\beta$'s endpoints, we can form the same contradiction as before that implies this path must have both odd and even length. In conclusion, one of ${\bf n_p^+}$ and ${\bf n_q^+}$ has all even coordinates and one does not.
  
Without loss of generality, assume that ${\bf n_p^+}$ is the endpoint with not all even coordinates. Then, ${\bf n_p^-}$ and ${\bf n_q^+}$ have even coordinates. Since Figure \ref{fig:nongeneric2} shows that one coordinate of ${\bf n_p^-}-{\bf n_q^+}$ is 1, we have a contradiction.

\end{proof}

Let $K$ be a lattice knot. Since $\rho_1$ is continuous on $\mathbb{L}^3$, $\rho_1$ is continuous on $K$ away from the diagonal $\{(x,x): x\in K\}$. On the diagonal, $\rho_1$ is defined to be one. For sequences $x_n\to x$ and $x'_n\to x$, we have $\rho_1(x_n,x'_n)\to1$ since  $d_K(y,z)= d_1(y,z)$ for all $y$ and $z$ in a $d_1$-neighborhood of radius less than one in $\mathbb{L}^3$. Thus, $\rho_1$ is continuous on lattice knots.

\begin{definition}
    The {\it Gromov $1$-distortion} of a lattice knot $K$ is 
    \[
        \delta^{(1)}(K) := \max_{a,b\in K} \rho_1({ a,b}).
    \]
\end{definition}

\noindent With $\rho_1$ continuous on $\mathbb{L}^3$ and $K$ compact, $\delta^{(1)}$ is well-defined as a maximum instead of a supremum.

Let $\mathcal{D}$ denote the set of dyadic rationals in $\mathbb{R}$, and define the set $\mathcal{D}_L$ as
\[
    \mathcal{D}_{{L}}: = (\mathcal{D} \times \mathbb{Z} \times \mathbb{Z}) \cup (\mathbb{Z} \times \mathcal{D} \times \mathbb{Z}) \cup (\mathbb{Z} \times \mathbb{Z} \times \mathcal{D}).
\]  The density of $\mathcal{D}$ in $\mathbb{R}$ implies that $\mathcal{D}_L$ is dense in $\mathbb{L}^3$. Furthermore, $\mathcal{D}_L \cap K$ is a dense subset of $K$.

\begin{proposition}\label{lem:3}
   For a lattice knot $K$, 

    \[
       \delta^{(1)}(K) = \delta_V(2K).
    \]
\end{proposition}

\begin{proof}

For a lattice knot $K$, $\delta_V(2K)$ is equivalent to \[ \max_{p,q\in V_m(K)} \rho_1(p,q).\] Since the Gromov 1-distortion of $K$ is equal to \[ \max_{a,b\in K} \rho_1(a,b) ,\] we have $\delta^{(1)}(K)\geq \delta_V(2K).$ Assume that $\delta^{(1)}(K) > \delta_V(2K)$, and let $a',b'\in K$ be the points realizing the Gromov 1-distortion maximum $\delta^{(1)}(K)=\rho_1(a',b')$. The points $a'$ and $b'$ cannot both be vertices of $K$ else $\delta^{(1)}(K) = \delta_V(2K).$

Since $\mathcal{D}_L\cap K$ is a dense subset of $K$, there exist sequences $a_n$ and $b_n$ in $ \mathcal{D}_L \cap K$ such that $a_n\to a'$ and $b_n\to b'$. 
For each $n$, let $m_n$ be the smallest non-negative integer such that \[ 2^{m_n}a_n,2^{m_n}b_n\in V(2^{m_n} K).\]   Then, $\delta_V(2^{m_n}K)=\delta_V(2K)$  by Lemma \ref{lem:lim2}.

The continuity of $\rho_1$ implies \begin{align*} \rho_1(a',b')&=\lim\limits_{n\to\infty} \rho_1(a_n,b_n)\\ &=\lim\limits_{n\to\infty}\rho_1(2^{m_n}a_n, 2^{m_n} b_n)\\ &\leq \lim\limits_{n\to\infty} \delta_V(2^{m_n}K)\\  &=\delta_V(2K),\end{align*}
contradicting the assumption $\rho_1(a',b')=\delta^{(1)}(K)>\delta_V(2K).$ The equality $\lim\limits_{n\to\infty} \delta_V(2^{m_n}K)=\delta_V(2K)$ holds since $ \delta_V(2^{m_n}K)=\delta_V(2K)$ for any $m_n>0$.

\end{proof}

This section concludes with two Gromov 1-distortion inequalities that will be used in the proof of Theorem \ref{fig:bound}.

\begin{lemma}
    \label{lem:2}
    For any lattice knot $K$,
    \[
 \delta^{(1)}(K)  \geq  \frac{1}{\sqrt{3}} \delta(K).
    \]
\end{lemma}

\begin{proof}

This follows from the inequality $ \sqrt3 \ d \geq d_1 .$

\end{proof}

\begin{corollary}
\label{lem:1distort}
For any lattice knot $K$,  we have \[\delta^{(1)}(K)\ge  \frac{5\pi}{3\sqrt3}.\]

\end{corollary}
\begin{proof}
This is a consequence of  Lemma \ref{lem:2} and Theorem \ref{thm:den}.  
\end{proof}


\section{Proof of Theorem \ref{fig:bound}}
\label{section:proof}

\noindent {\bf Theorem \ref{fig:bound}}
 Let $K$ be a non-trivial lattice knot. Then,   \[
        \delta_V(K) \geq \frac{5\pi}{3\sqrt3}-1.
    \]

\begin{proof}

If $\delta_V(2K)=\delta_V(K)$, then  Proposition \ref{lem:3} and Corollary \ref{lem:1distort} give

    \[
          \delta_V(K) = \delta_V(2K) = \delta^{(1)}(K) \geq \frac{5\pi}{3\sqrt3}.    
    \]

We finish the proof with the case $\delta_V(2K)>\delta_V(K).$ Lemma \ref{lem:bigger} tells us that there exists antipodal, non-generic midpoints $p$ and $q$ of $K$ such that their neighbors satisfy the ordering of (\ref{3}) and  \[ \rho_1(p,q)=\delta_V(2K)>\delta_V(K).\] Such midpoints are illustrated in Figure \ref{fig:nongeneric2} where we see that there exist neighbor vertices ${\bf n_p}$ and ${\bf n_q}$ such that  \[\rho_1({\bf n_{p}, n_{q}})=\frac{d_K(p,q)-1}{d_1(p,q)}. \]

The initial argument of this proof tells us that \[ \rho_1(p,q)=\delta_V(2K)=\delta^{(1)}(K)\geq\frac{5\pi}{3\sqrt3}.\]
Thus,

\[   \delta_V(K)\geq \rho_1({\bf n_{p},n_{q} }) = \frac{d_K(p, q) - 1}{d_1(p, q)} \geq \delta_V(2K)-1\geq \frac{5\pi}{3\sqrt3}-1.\]

\end{proof}

\begin{corollary}

If $\delta_V[K]=1$, then $[K]$ is trivial

\end{corollary}

\begin{proof}
Suppose that $[K]$ is not trivial and $\delta_V[K]=1$. Then, there exists a sequence of conformations $K_n$ such that $\delta_V(K_n)\to 1$ as $n\to \infty$. Some term $K_m$ of this sequence satisfies $\delta_V(K_m)<\frac{5\pi}{3\sqrt3}-1$. Theorem \ref{fig:bound} implies that $K_m$ is unknotted, contradicting the non-triviality of $[K]$. 

\end{proof}

Therefore, $\delta_V[K]=1$ if and only if $[K]$ is trivial.

\section{Proof of Theorem \ref{prop:bounding1}}
\label{section:misc}

\noindent {\bf Theorem \ref{prop:bounding1}} For any lattice knot $K$, 
\[ \delta(K)\geq \delta_V(K)\geq \delta(K)/\sqrt3-1.\]

\begin{proof}
Since $d_1\geq d$, we have $\delta(K)\geq \delta_V(K)$.  Proposition \ref{lem:3} and Lemma \ref{lem:2} imply that \[ \delta_V(2K) = \delta^{(1)}(K) \geq \delta(K)/\sqrt3.\]
If $\delta_V(2K)=\delta_V(K)$, then the proof is complete. If $\delta_V(2K)>\delta_V(K),$ then the last line of Theorem \ref{fig:bound}'s proof shows that \[\delta_V(K)\geq \delta(2K)-1.\]     
\noindent Thus,
\[\delta_V(K)+1\geq \delta(K)/\sqrt3.\]

\end{proof}

The preceding theorem implies that a bound on Gromov distortion carries to vertex distortion. This allows for a vertex distortion port of Pardon's result on Gromov distortion.

\begin{theorem}[Pardon '11 \cite{Pardon_2011}]
For the $(p,q)$-torus knot $T_{p,q}$, \[ \delta[T_{p,q}]\geq \frac{\min(p,q)}{160}.\]
\end{theorem}

\begin{corollary}
For the $(p,q)$-torus knot $T_{p,q}$, \[ \delta_V[T_{p,q}]\geq \frac{\min(p,q)}{160\sqrt3}-1.\]
\label{cor:pardon}
\end{corollary} 

Conjecture 1 of \cite{Campisi_2021} purported, \[ \delta_V[T_{p,q}]\to\infty \ \text{as} \ p,q\to\infty.\] Corollary \ref{cor:pardon} affirms this conjecture.

\begin{figure}[ht]
\centering

\begin{overpic}[unit=.5mm,scale=.25]{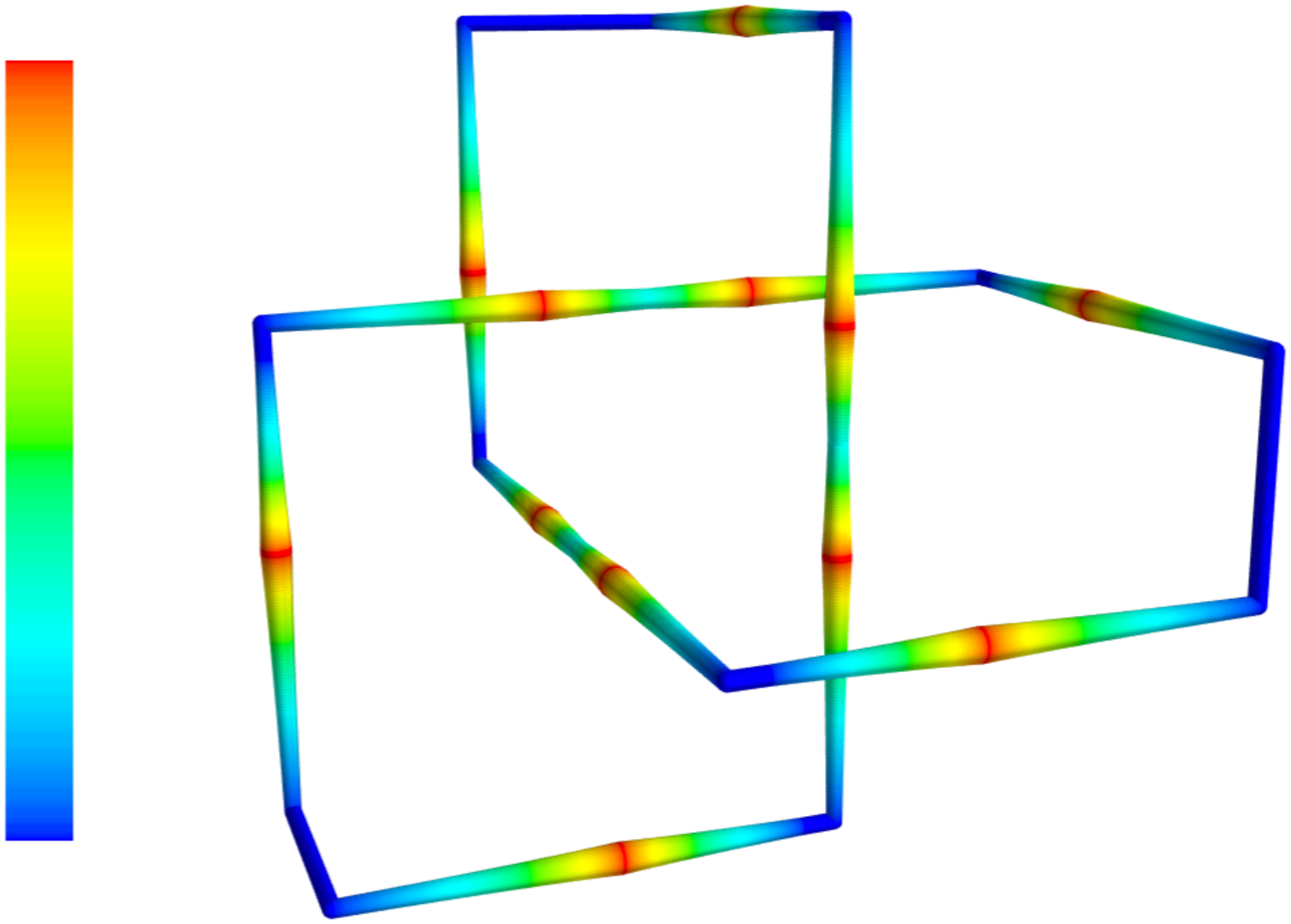}\put(-10,93){9.90}\put(-10,11){5.45}\put(-10,53){7.82}

\end{overpic}
\caption{Trefoil colored by $\rho_1(x,\cdot)$'s maximum at $x$.}
\label{fig:trefoil}
\end{figure}

 Our vertex distortion calculator is available here: \\  \url{https://github.com/sava-1729/lattice-knots-vis}. We conclude with the program's algorithm.

\vspace{4mm}
\RestyleAlgo{ruled}

\SetKwComment{Comment}{/* }{ */}

\begin{algorithm}[H]
\caption{Computing the Vertex Distortion of a Lattice Knot}\label{alg:two}
\KwData{$(\mathbf{v_n}\ |\ 1 \leq n \leq N$), the ordered sequence of vertices of an oriented knot $K$}
\KwResult{$\delta_V(K) = \max_{1\leq i,j\leq N} \rho_1(\mathbf{v_i}, \mathbf{v_j}) = \delta$, $V_\delta(K) = \{\{\mathbf{p}, \mathbf{q}\} \in V(K) \times V(K)\ |\ \rho_1(\mathbf{p}, \mathbf{q}) = \delta_V(K)\} = S$}
$\delta \gets 1$\;
$S \gets \varnothing$\;
\For{$d_K \leftarrow \lfloor \nicefrac{N}{2} \rfloor$ \KwTo $1$}{
    \For{$i \leftarrow 0$ \KwTo $N-1$}{
        $j \gets i - d_K \pmod{N}$\;
        $(x_1,y_1,z_1) \gets \mathbf{v_i}$\;
        $(x_2,y_2,z_2) \gets \mathbf{v_j}$\;
        $d_1 \gets |x_1-x_2| + |y_1-y_2| + |z_1-z_2|$\;
        \If{$d_1 > 0$}{
            $\rho_1 \gets \nicefrac{d_K}{d_1}$\;
            \If{$\rho_1 > \delta$}{
                $\delta \gets \rho_1$\;
                $S \gets \{\{\mathbf{v_i}, \mathbf{v_j}\}\}$\;
            }
            \ElseIf{$\rho_1 = \delta$ {\bf and} $\{\mathbf{v_i}, \mathbf{v_j}\} \notin S$}{
                $S \gets S \cup \{\{\mathbf{v_i}, \mathbf{v_j}\}\}$\;
            }
        }
    }

}

\end{algorithm}

\bibliographystyle{hplain}   
    \bibliography{vertexunknot}

\begin{bibdiv}
\begin{biblist}

\bib{Blair_2020}{article}{
      author={Blair, Ryan},
      author={Campisi, Marion},
      author={Taylor, Scott~A.},
      author={Tomova, Maggy},
       title={Distortion and the bridge distance of knots},
        date={2020},
     journal={Journal of Topology},
      volume={13},
      number={2},
       pages={669\ndash 682},
}

\bib{Campisi_2021}{article}{
      author={Campisi, Marion},
      author={Cazet, Nicholas},
       title={Vertex distortion of lattice knots},
        date={2021},
     journal={Journal of Knot Theory and Its Ramifications},
      volume={30},
      number={07},
       pages={2150043},
}

\bib{Denne_2004}{article}{
      author={Denne, Elizabeth},
      author={Sullivan, John~M.},
       title={The distortion of a knotted curve},
        date={2009},
        ISSN={00029939, 10886826},
     journal={Proceedings of the American Mathematical Society},
      volume={137},
      number={3},
       pages={1139\ndash 1148},
         url={http://www.jstor.org/stable/20535841},
}

\bib{gromov1983}{article}{
      author={Gromov, Mikhael},
       title={Filling riemannian manifolds},
        date={1983},
     journal={Journal of Differential Geometry},
      volume={18},
      number={1},
       pages={1 \ndash  147},
         url={https://doi.org/10.4310/jdg/1214509283},
}

\bib{Pardon_2011}{article}{
      author={Pardon, John},
       title={On the distortion of knots on embedded surfaces},
        date={2011},
        ISSN={0003-486X},
     journal={Annals of Mathematics},
      volume={174},
      number={1},
       pages={637\ndash 646},
         url={http://dx.doi.org/10.4007/annals.2011.174.1.21},
}

\end{biblist}
\end{bibdiv}

\end{document}